\documentclass[a4paper,11pt,leqno]{amsart}

\usepackage[hmargin=2.0cm,vmargin=2.0cm,a4paper,dvips]{geometry}
\usepackage[T1]{fontenc}
\usepackage[light,slantedGreeks
 ]{kpfonts}

\usepackage[latin1]{inputenc}
\usepackage[english]{babel}
\usepackage{latexsym}
\usepackage{amsmath}
\usepackage{amsfonts}
\usepackage{amssymb}
\usepackage{amsthm}
\usepackage{amscd}
\usepackage{enumerate}
\usepackage{comment}
\usepackage{microtype}

\theoremstyle{definition}
\newtheorem{dfn}{DEFINITION}[section]

\newtheorem{cnv}[dfn]{CONVENTION}

\newtheorem{rec+dfn}[dfn]{REMINDER \& DEFINITION}
\newtheorem{rem}[dfn]{REMARK}
\newtheorem{ntn}[dfn]{NOTATION}
\newtheorem{ntn+rem}[dfn]{NOTATION \& REMARK}
\newtheorem{exa}[dfn]{EXAMPLE}
\theoremstyle{plain}
\newtheorem{thm}[dfn]{THEOREM}
\newtheorem{lem}[dfn]{LEMMA}
\newtheorem{cor}[dfn]{COROLLARY}
\newtheorem{pro}[dfn]{PROPOSITION}

\DeclareMathAlphabet{\mathpzc}{OT1}{pzc}{m}{it}

\DeclareMathOperator{\MO}{MO}
\DeclareMathOperator{\NO}{NO}
\DeclareMathOperator{\SO}{SO}
\DeclareMathOperator{\TO}{TO}

\DeclareMathOperator{\ls}{ls}
\DeclareMathOperator{\LS}{LS}
\DeclareMathOperator{\lt}{lm}
\DeclareMathOperator{\LT}{LM}
\DeclareMathOperator{\Img}{Im}

\DeclareMathOperator{\Supp}{Supp}
\DeclareMathOperator{\supp}{supp}

\renewcommand{\epsilon}{\varepsilon}
\renewcommand{\varsigma}{\sigma}
\renewcommand{\imath}{\iota}

\renewcommand{\Im}{\Img}

\newcommand{\NN}{\mathbb{N}}

\newcommand{\RR}{\mathbb{R}}

\newcommand{\dd}{\partial}
\newcommand{\csi}{\xi}
\newcommand{\wo}{\smallsetminus}
\renewcommand{\setminus}{\smallsetminus}
\renewcommand{\emptyset}{\varnothing}
\newcommand{\concept}[1]{\textbf{\textit{#1}}}

\newcommand{\res}{ \text{$\upharpoonright$} }
\newcommand{\und}{\,\mathop{\wedge}\,}
\newcommand{\oder}{\,\mathop{\vee}\,}
\newcommand{\dann}{\Rightarrow}

\newcommand{\gdw}{\Leftrightarrow}

\begin{document}

\relpenalty 9999 
\binoppenalty 10000

\begin{center}
{\Large\textbf{UNIVERSAL GR\"OBNER BASES IN WEYL ALGEBRAS}}
\normalsize
\vskip\baselineskip
\textnormal{ROBERTO BOLDINI}
\vskip\baselineskip
\textnormal{INSTITUTE OF MATHEMATICS OF THE UNIVERSITY OF ZURICH}
\footnote{My gratitude to Prof.~em.~Dr.~Markus Brodmann and Prof.~Dr.~Joseph Ayoub, University of Zurich.}
\vskip\baselineskip
\textnormal{NOVEMBER 2009}
\end{center}

\vskip 1.65\baselineskip





\begin{center}
\begin{minipage}{0.69\linewidth}
\Small{
A topological space $\TO(S)$ of total orderings on any given set $S$ is introduced
and it is shown that $\TO(S)$ is compact. 
The set $\NO(N)$ of all normal orderings of the $n^\text{th}$ Weyl algebra $W$
is a closed subspace of $\TO(N)$,
where $N$ is the set of all normal monomials of $W$.
Hence $\NO(N)$ is compact and, as a consequence of this fact and by a division theorem valid in $W$,
we give a proof that each left ideal of $W$ admits a universal Gr\"obner basis.
These notes have been inspired by the beautiful article \cite{Sik} of A.~Sikora.}
\end{minipage}
\end{center}

\vskip 2\baselineskip


\renewcommand{\contentsname}{} 


\section{TOPOLOGICAL SPACES OF TOTAL ORDERINGS OF SETS}\label{sec1}

\noindent
In this section, let $S$ be a set.

\begin{dfn}
A \concept{total ordering} on $S$ is a binary relation $\preceq$ on $S$ such that
for all $a,b,c\in S$ it holds
antisymmetry: $a \preceq b \und b \preceq a \dann a = b$, 
transitivity: $a \preceq b \und b \preceq c \dann a \preceq c$, 
totality: $a \preceq b \oder b \preceq a$. 
Totality implies reflexivity: $a \preceq a$ for all $a\in S$.
The non-empty set of all total orderings on $S$ is denoted by $\TO(S)$.

Given any ordered pair $(a,b)\in S\times S$,
let $\mathfrak{U}_{(a,b)}$ be the set of all total orderings $\preceq$ on $S$ for which $a\preceq b$.
Let $\mathcal{U}$ be the coarsest topology on $\TO(S)$ for which all the sets $\mathfrak{U}_{(a,b)}$ are open.
This is the topology
for which $\{\mathfrak{U}_{(a,b)}\mid (a,b)\in S\times S\}$ is a subbasis, 
i.e.~the open sets in $\mathcal{U}$ are precisely the unions of finite intersections of sets
of the form $\mathfrak{U}_{(a,b)}$.
Observe that $\mathfrak{U}_{(a,a)}=\TO(S)$ and that $\mathfrak{U}_{(a,b)}=\TO(S)\wo \mathfrak{U}_{(b,a)}$ if~$a\ne b$,
so that the open sets $\mathfrak{U}_{(a,b)}$ are also closed.

Let $\mathbf{S}$ be any \concept{filtration} of $S$,
i.e.~a family $\mathbf{S}:=(S_i)_{i\in\NN_0}$ of subsets $S_i$ of $S$ with
${S_{0}=\emptyset}$,
${S_{i}\subseteq S_{i+1}}$ for all ${i\in\NN_0}$ and
$S=\bigcup_{i\in\NN_0}S_i$.
We define the function
$d_\mathbf{S}:\TO(S)\times\TO(S)\rightarrow\RR$ by putting
${d_\mathbf{S}(\preceq',\preceq''):= 2^{-r}}$ with
${r:=\sup{}\{i\in\NN_0\mid\mathrm{\preceq'}\res_{S_i}=\mathrm{\preceq''}\res_{S_i}\}}$,
where $\mathrm{\res}$ denotes restriction.
It holds $\{0\} \subseteq \Im(d_\mathbf{S}) \subseteq [0,1]$.
As $\mathbf{S}$ is exhaustive,
we have that ${d_\mathbf{S}(\preceq',\preceq'')=0}$ if and only if $\mathrm{\preceq'}=\mathrm{\preceq''}$.
Obviously, ${d_\mathbf{S}(\preceq',\preceq'')}={d_\mathbf{S}(\preceq'',\preceq')}$.
Finally, 
${d_\mathbf{S}(\preceq',\preceq''')}\le {d_\mathbf{S}(\preceq',\preceq'')}+{d_\mathbf{S}(\preceq'',\preceq''')}$,
because
${d_\mathbf{S}(\preceq',\preceq''')}\le\max{}\{{d_\mathbf{S}(\preceq',\preceq'')},{d_\mathbf{S}(\preceq'',\preceq''')}\}$.
Thus $d_\mathbf{S}$ is a metric on $\TO(S)$, dependent on the choice of the filtration $\mathbf{S}$ of $S$.
\end{dfn}

\begin{thm}\label{U=N}
Let $\mathbf{S}=(S_i)_{i\in\NN_0}$ be any filtration of $S$ such that each of the sets $S_i$ is finite.
Let $\mathcal{N}$ be the topology on $\TO(S)$ induced by the metric $d_\mathbf{S}$, 
i.e.~$\mathfrak{N}\in\mathcal{N}$ if and only if $\,\mathfrak{N}$ is a union of finite intersections of sets of the form
$\mathfrak{N}_{r}(\preceq):=\{\preceq'\in\TO(S)\mid d_\mathbf{S}(\preceq,\preceq')<2^{-r}\}$ with $r\in\NN_0$
and $\mathrm{\preceq}\in\TO(S)$.
Then $\mathcal{N}=\mathcal{U}$, in particular the topology $\mathcal{N}$ is independent of the 
chosen filtration $\mathbf{S}$ of $S$.
\end{thm}

\begin{proof}
Let $r\in\NN_0$ and $\mathrm{\preceq} \in \TO(S)$.
We claim that $\mathfrak{N}_{r}(\preceq)\in\mathcal{U}$.
Let $\mathfrak{U}:=\bigcap_{(a,b)}\mathfrak{U}_{(a,b)}$,
where the intersection is taken over all ordered pairs $(a,b)\in S_{r+1}\times S_{r+1}$
with $a\preceq b$.
Then $\mathrm{\preceq}\in \mathfrak{U}\in\mathcal{U}$.
Hence $\mathrm{\preceq}'\in \mathfrak{N}_{r}(\preceq)$ if and only if
$\mathrm{\preceq'}\res_{S_{r+1}}\!=\,\mathrm{\preceq}\res_{S_{r+1}}$,
and this is the case if and only if it holds $a\preceq' b\gdw a\preceq b$ for all $(a,b)\in S_{r+1} \times S_{r+1}$,
which is true if and only if $\mathrm{\preceq'}\in \mathfrak{U}$.
Thus $\mathfrak{N}_r(\preceq)=\mathfrak{U}$, and this shows that $\mathcal{N}\subseteq\mathcal{U}$.

On the other hand, let $(a,b)\in S\times S$ be any ordered pair. 
We claim that the set $\mathfrak{U}_{(a,b)}$ is open with respect to the metric $d_\mathbf{S}$.
Let $\mathrm{\preceq}\in \mathfrak{U}_{(a,b)}$, so that $a\preceq b$.
We find $r\in\NN_0$ such that $(a,b)\in S_{r+1}\times S_{r+1}$.
If $\mathrm{\preceq'}\in \mathfrak{N}_r(\preceq)$,
then $\mathrm{\preceq'}\res_{S_{r+1}}\! = \mathrm{\preceq}\res_{S_{r+1}}$,
in particular $a\preceq' b$, so that $\mathrm{\preceq'}\in \mathfrak{U}_{(a,b)}$,
thus $\mathfrak{N}_r(\preceq)\subseteq \mathfrak{U}_{(a,b)}$.
Hence $\mathfrak{U}_{(a,b)}$ is open with respect to $\mathcal{N}$,
and we conclude that $\mathcal{U}\subseteq\mathcal{N}$.
\end{proof}

\begin{cnv}
From now on, when we say ``the topological space $\TO(S)$'',
we always mean $\TO(S)$ endowed with the topology $\mathcal{U}$.
With ``topological subspace of $\TO(S)$'' we always intend a subset of $\TO(S)$
provided with its relative topology with respect to~$\mathcal{U}$.
\end{cnv}

\begin{dfn}\label{axiom}
An \concept{ultrafilter} $U$ on a given set $X$ is a family of subsets of $X$ such that
(a)~$\emptyset\notin U$,
(b)~$A\subseteq B\subseteq X\und A\in U\dann B\in U$,
(c)~$A\in U\und B\in U\dann A\cap B\in U$,
(d)~$A\subseteq X\dann A\in U \oder X\wo A\in U$.
\end{dfn}

\begin{thm}\label{compact}
The topological space $\TO(S)$ is compact.%
\footnote{The existence and proof of Theorem~\ref{compact} were kindly communicated to the 
author by Prof.~Dr.~Matthias Aschenbrenner, University of California, Los Angeles.}
\end{thm}

\begin{proof}
Suppose by contradiction that $\TO(S)$ is not compact.
Then we find an infinite index set $I$ and families $(a_i)_{i\in I}$ and $(b_i)_{i\in I}$
of elements $a_i,b_i\in S$ such that $(\mathfrak{U}_{(a_i,b_i)})_{i\in I}$ is a covering of
$\TO(S)$ which admits no finite subcovering.
Thus for each finite subset $s\subseteq I$ there exists ${\preceq_s}\in\TO(S)$ such that
${\preceq_s}\notin\bigcup_{i\in s}\mathfrak{U}_{(a_i,b_i)}$,
i.e.~for all $i\in s$ it holds $a_i\succ_s b_i$.

Let $I^\ast$ be the set of all non-empty finite subsets of $I$.
For each $s\in I^\ast$ let $s^\ast := \{t\in I^\ast\mid s\subseteq t\}$.
As $s_1^\ast\cap s_2^\ast=(s_1\cup s_2)^\ast$ for all $s_1,s_2\in I^\ast$,
the family $(s^\ast\mid s\in I^\ast)$ is closed under finite intersections.
Hence, by the Ultrafilter Lemma, there exists an ultrafilter $U$ on $I^\ast$ such that
$s^\ast\in U$ for all $s\in I^\ast$.

Let ${\preceq}$ be a binary relation on $S$ defined by
$a\preceq b :\Leftrightarrow \{s\in I^\ast\mid a\preceq_s b\}\in U$.
By axioms (c) and (a) of \ref{axiom}, ${\preceq}$ is antisymmetric.
By axioms (c) and (b) of \ref{axiom}, ${\preceq}$ is transitive.
By axioms (d) and (b) of \ref{axiom}, ${\preceq}$ is total.
So ${\preceq}\in\TO(S)$.
On the other hand, by our choice of the orderings ${\preceq_s}$,
it holds $a_i\succ b_i$ for all $i\in I$,
thus ${\preceq}\notin\bigcup_{i\in I}\mathfrak{U}_{(a_i,b_i)}=\TO(S)$, a contradiction.
\end{proof}

\begin{thm}\label{closed}
For each $a\in S$ the set
${}\SO_{a}(S):=\{\mathrm{\preceq}\in\TO(S) \mid \forall\,b\in S :a\preceq b\}$
is closed in $\TO(S)$.
Hence 
the topological subspace $\SO_{a}(S)$ of ${}\TO(S)$ is compact.
\end{thm}

\begin{proof}
It holds $\SO_{a}(S)=\bigcap_{b\in S}\mathfrak{U}_{(a,b)}=\bigcap_{b\in S\wo\{a\}}\mathfrak{U}_{(a,b)}
=\TO(S)\wo\bigcup_{b\in S\wo\{a\}}\mathfrak{U}_{(b,a)}$, thus $\SO_a(S)$ is closed in $\TO(S)$.
As $\TO(S)$ is compact by \ref{compact},
the closedness of $\SO_a(S)$ in $\TO(S)$ implies that $\SO_a(S)$ with the relative topology is 
compact.
\end{proof}

\vskip\baselineskip

\section{UNIVERSAL GR\"OBNER BASES IN WEYL ALGEBRAS}\label{sec3}

\noindent
In this section, let $K$ be a field of characteristic zero and $n\in\NN$.

We denote by $W$ the $n^\text{th}$ Weyl algebra $K\langle \xi_1,\ldots,\xi_n,\dd_1,\ldots,\dd_n\rangle$ with
${\xi_i\xi_j-\xi_j\xi_i=0}$ and ${\dd_i\dd_j-\dd_j\dd_i=0}$ and ${\dd_i\xi_j-\xi_j\dd_i=\delta_{ij}}$ over $K$,
where $\delta_{ij}\in K$ is the Kronecker delta.
We remind that $W$ is a central simple left noetherian $K$-algebra and a domain, see~\cite{Bor} and~\cite{Cou}.

We write $K[X,Y]$ for the commutative polynomial ring over $K$ in the indeterminates $X_1,\ldots,X_n$ and $Y_1,\ldots,Y_n$.

\begin{rec+dfn}\label{Supp}
The countable set $N:=\{\xi^\lambda\dd^\mu\mid(\lambda,\mu)\in\NN_0^n\times\NN_0^n\}$ of the
\concept{normal monomials} of $W$ is a  basis of the free $K$-module $W$.

Thus each  $w\in W$ can be  written  in \concept{canonical form}
as a sum $\sum_{(\lambda,\mu)\in\supp(w)}c_{(\lambda,\mu)}\xi^\lambda\dd^\mu$
for a uniquely determined finite subset $\supp(w)$ of $\NN_0^n\times\NN_0^n$ such that $c_{(\lambda,\mu)}\in K\wo\{0\}$
for all $(\lambda,\mu)\in\supp(w)$.

For each $w\in W$ we define the set $\Supp(w):=\{\csi^\lambda\dd^\mu\in N\mid (\lambda,\mu)\in\supp(w)\}$, which we call
the \concept{support} of $w$.
For each subset $V$ of $W$ we put $\Supp(V):=\bigcup_{v\in V}\Supp(v)$.
\end{rec+dfn}

\begin{dfn}
A \concept{normal  ordering} of $W$,
or \concept{term ordering} of $W$ according to the definition  in~\cite{Sai},
is a total ordering $\preceq$ on $N$ such that
for all $\lambda,\mu,\rho,\sigma,\alpha,\beta\in\NN_0^n$
it holds
well-foundedness:~${1\preceq \xi^\lambda\dd^\mu}$,
and
compatibility:~${\xi^\lambda\dd^\mu\preceq\xi^\rho\dd^\sigma
\dann
\xi^{\lambda+\alpha}\dd^{\mu+\beta}\preceq\xi^{\rho+\alpha}\dd^{\sigma+\beta}}$.
Since $N$ is a $K$-basis of $W$, these axioms are e\-qui\-va\-lent to:
${1\prec \xi^\lambda\dd^\mu}$ for all $\lambda,\mu\in\NN_0^n$ with $(\lambda,\mu)\ne (0,0)$,
and
${\xi^\lambda\dd^\mu\prec\xi^\rho\dd^\sigma
\dann
\xi^{\lambda+\alpha}\dd^{\mu+\beta}\prec\xi^{\rho+\alpha}\dd^{\sigma+\beta}}$
for all $\lambda,\mu,\rho,\sigma,\alpha,\beta\in\NN_0^n$.
The set of all normal orderings of $W$ is denoted by $\NO(N)$.
\end{dfn}

\begin{exa}
The \concept{lexicographical ordering} $\mathrm{\preceq_\mathrm{lex}}$ on $N$
defined by
\[
{\xi^\lambda\dd^\mu\preceq_\mathrm{lex}\xi^\rho\dd^\sigma }
\gdw
{(\lambda = \rho \und \mu = \sigma)} 
\oder
{(\lambda = \rho \und \mu \ne \sigma \und \mu_{i( \mu,\sigma )} < \sigma_{i( \mu,\sigma )})}
\oder
{(\lambda \ne \rho \und \lambda_{i( \lambda,\rho )} < \rho_{i( \lambda,\rho )})}
\]
for all $\lambda,\mu,\rho,\sigma\in\NN_0^n$,
where we put $i(\alpha,\beta):=\min{}\{\,j\mid 1\le j\le n \,\wedge\, \alpha_j\ne\beta_j\,\}$
for all $\alpha,\beta\in\NN_0^n$ with $\alpha\ne\beta$,
is a normal ordering of $W$. 
\end{exa}

\begin{rec+dfn}\label{Supp2}
The countable set $M:=\{X^\lambda Y^\mu\mid(\lambda,\mu)\in\NN_0^n\times\NN_0^n\}$ of the
\concept{monomials} of $K[X,Y]$ is a basis of the free $K$-module $K[X,Y]$.

Thus each  $p\in K[X,Y]$ can be  written  in \concept{canonical form}
as a sum $\sum_{(\lambda,\mu)\in\supp(p)}c_{(\lambda,\mu)}X^\lambda Y^\mu$
for a uniquely determined finite subset $\supp(p)$ of $\NN_0^n\times\NN_0^n$ such that $c_{(\lambda,\mu)}\in K\wo\{0\}$
for all $(\lambda,\mu)\in\supp(p)$.

For each $p\in K[X,Y]$ we define the set $\Supp(p):=\{X^\lambda Y^\mu\in M\mid (\lambda,\mu)\in\supp(p)\}$, which we call
the \concept{support} of $p$.
For each subset $U$ of $K[X,Y]$ we put $\Supp(U):=\bigcup_{u\in U}\Supp(u)$.
\end{rec+dfn}

\begin{dfn}
According to~\cite{Wat} and equivalently to~\cite{Cox},
a \concept{monomial ordering} of $K[X,Y]$
is a total ordering $\leq$ on $M$ such that
for all $\lambda,\mu,\rho,\sigma,\alpha,\beta\in\NN_0^n$
it holds
well-foundedness:~${1\leq X^\lambda Y^\mu}$,
and
compatibility:~${X^\lambda Y^\mu\leq X^\rho Y^\sigma
\dann
X^{\lambda+\alpha}Y^{\mu+\beta}\leq X^{\rho+\alpha}Y^{\sigma+\beta}}$.
Equivalently:
 ${1< X^\lambda Y^\mu}$ for all $\lambda,\mu\in\NN_0^n$ with $(\lambda,\mu)\ne(0,0)$,
and
 ${X^\lambda Y^\mu < X^\rho Y^\sigma
\dann
X^{\lambda+\alpha}Y^{\mu+\beta} < X^{\rho+\alpha}Y^{\sigma+\beta}}$
for all $\lambda,\mu,\rho,\sigma,\alpha,\beta\in\NN_0^n$.
The set of all monomial orderings of $K[X,Y]$ is denoted by $\MO(M)$.
\end{dfn}

\begin{rem}\label{Phi}
There exists an isomorphism of $K$-modules $\Phi:W\rightarrow K[X,Y]$ which maps the basis $N$ of $W$ to the basis
$M$ of $K[X,Y]$ by the rule $\xi^\lambda\dd^\mu\mapsto X^\lambda Y^\mu$.
\end{rem}

\begin{thm}\label{phi}
The $K$-isomorphism $\Phi$ induces a homeomorphism $\phi:\NO(N)\rightarrow\MO(M)$ given by ${{\preceq}\mapsto{\leq}}$
where ${\leq}$ is defined by
$X^\lambda Y^\mu \leq X^\rho Y^\sigma :\Leftrightarrow \Phi^{-1}(X^\lambda Y^\mu) \preceq \Phi^{-1}(X^\rho Y^\sigma)$.
\end{thm}

\begin{proof}
It is immediate to check that $\phi(\preceq)$ is indeed a monomial ordering of $K[X,Y]$ for each 
normal ordering ${\preceq}$ of $W$, thus $\phi$ is well-defined.

Let ${\leq}\in\Im(\phi)$.
Suppose that there exist two distint
$\mathrm{\preceq}, \mathrm{\preceq'} \in \NO(N)$ such that $\phi(\preceq)={\leq}=\phi(\preceq')$.
There exist normal monomials $u,v\in N$ such that $u\succ v$ and $u\preceq' v$.
Then $v\preceq u$ by totality.
Hence $\Phi(v)\leq\Phi(u)$ and $\Phi(u)\leq\Phi(v)$.
So $\Phi(u)=\Phi(v)$ by antisymmetry.
It follows $u=v$ as $\Phi$ is injective.
But this contradicts the reflexivity of ${\preceq}$.
Therefore $\phi$ is injective.

Now let $\mathrm{\leq}\in\MO(M)$.
We define a binary relation $\mathrm{\preceq}$ on $N$ by setting $u\preceq v :\Leftrightarrow \Phi(u)\leq\Phi(v)$
for all $u,v\in N$.
One easily verifies that ${\preceq}\in\NO(N)$ and that $\phi(\preceq)={\leq}$, thus $\phi$ is surjective.

We check now that $\phi$ is continuous.
As $\phi$ is bijective, so that $\phi$ commutes with  intersections and unions,
it is sufficient to consider any  subbases of $\MO(M)$ and $\NO(N)$.
So, let $(a,b)\in M\times M$ and consider the open subset $\mathfrak{U}_{(a,b)}\cap\MO(M)$ of $\MO(M)$.
As
$\phi^{-1}(\mathfrak{U}_{(a,b)}\cap\MO(M))=\mathfrak{U}_{(\Phi^{-1}(a),\Phi^{-1}(b))}\cap\NO(N)$,
 $\phi$ is continuous.

Similarly, for all $(u,v)\in N\times N$ one has
$\phi(\mathfrak{U}_{(u,v)}\cap\NO(N))=\mathfrak{U}_{(\Phi(u),\Phi(v))}\cap\MO(M)$,
so that $\phi$ is open
and hence $\phi^{-1}$ is continuous.
\end{proof}

\begin{thm}\label{NO(N)}
In the notation of~\ref{closed}, $\NO(N)$ is a closed subset of ${}\SO_1(N)$.
Hence $\NO(N)$ is a compact topological subspace of ${}\SO_1(N)$.
\end{thm}

\begin{proof}
$\NO(N)$ is of course a subset of $\SO_1(N)$.
Let $(S_i)_{i\in\NN_0}$ be a filtration of $N$ with finite sets $S_i$.
Let $\mathrm{\preceq}\in\SO_1(N)$ be an accumulation point of $\NO(N)$.
Thus for each $r\in\NN_0$ there exists ${\mathrm{\preceq}_r\in\NO(N)\wo\{\mathrm{\preceq}\}}$ with
${\mathrm{\preceq}_r\in \mathfrak{N}_r(\preceq)\cap\SO_1(N)}$, so  $\mathrm{\preceq}_r$ and $\mathrm{\preceq}$
agree on $S_{r+1}$.
Choose any ${\lambda,\mu,\rho,\sigma,\alpha,\beta\in\NN_0^n}$ and
assume that, say, $\xi^\lambda\dd^\mu\preceq\xi^\rho\dd^\sigma$.
We find $r\in\NN_0$ such that
$\xi^\lambda\dd^\mu,\xi^\rho\dd^\sigma,\xi^{\lambda+\alpha}\dd^{\mu+\beta},\xi^{\rho+\alpha}\dd^{\sigma+\beta}\in S_{r+1}$.
There exists $\mathrm{\preceq}_r$ as above such that $\xi^\lambda\dd^\mu\preceq_r\xi^\rho\dd^\sigma$.
Since $\mathrm{\preceq}_r$ is a normal ordering of $W$, it follows
$\xi^{\lambda+\alpha}\dd^{\mu+\beta}\preceq_r\xi^{\rho+\alpha}\dd^{\sigma+\beta}$.
Therefore 
${\xi^{\lambda+\alpha}\dd^{\mu+\beta}\preceq\xi^{\rho+\alpha}\dd^{\sigma+\beta}}$.
Thus $\mathrm{\preceq}\in\NO(N)$.
Hence $\NO(N)$ contains all its accumulation points in $\SO_1(N)$ and therefore $\NO(N)$ is closed in $\SO_1(N)$.
Compactness of $\NO(N)$ follows now from closedness of $\NO(N)$ in the compact space $\SO_1(N)$, see~\ref{closed}.
\end{proof}

\begin{cor}
$\MO(M)$ is compact.
\end{cor}

\begin{proof}
Clear by~\ref{phi} and~\ref{NO(N)}.
\end{proof}

\begin{cor}
Let ${\preceq}\in\NO(N)$ and put ${\leq}:=\phi(\preceq)$.
Then  ${\preceq}$ and ${\leq}$  are  well-orderings.
\end{cor}

\begin{proof}
Since~${\leq}$ is a monomial ordering of $K[X,Y]$ by~\ref{phi}, one has that ${\leq}$ is a well-ordering on $M$,
as it is well-known from Commutative Algebra, see e.g.~\cite[Theorem~15.1]{Wat}.
Let $V$ be a non-empty subset of~$N$.
Then $U:=\Phi(V)$ is a non-empty subset of $M$ and it exists therefore some $u_0\in U$ such that $u_0\leq u$ for all $u\in U$.
With $v_0:=\Phi^{-1}(u_0)$ it follows $v_0\preceq v$ for all $v\in V$.
Thus ${\preceq}$ is a well-ordering on~$N$.
\end{proof}

\newcommand{\bucato}{+}

\begin{ntn}
For each subset $B$ of any additive monoid  $(A,+,0)$ we  write $B^{\bucato}_{}$ for $B\setminus\{0\}$.

For each ${\leq}\in\MO(M)$ and each $p\in K[X,Y]^\bucato_{}$ we write $\LT_\leq(p)$ for the greatest monomial
in the canonical form of $p$ with respect to~${\leq}$.

For each $\mathrm{\preceq}\in\NO(N)$ and each $w \in W^{\bucato}_{}$
we write $\lt_\preceq(w)$ for the greatest normal monomial in the canonical form of $w$
with respect to~$\mathrm{\preceq}$
and denote $\Phi(\lt_\preceq(w))$ by $\LT_\preceq(w)$.

Let ${\preceq}\in\NO(N)$ and $w\in W^\bucato\!$.
If $c\in K^\bucato\!$ is the coefficient of $\lt_\preceq(w)$ in the canonical form of $w$,
we write $\ls_\preceq(w)$ for  $c\lt_\preceq(w)$ and $\LS_\preceq(w)$ for $\Phi(\ls_\preceq(w))$.
Whenever $u,v\in W^\bucato\!$ and $\ls_\preceq(u)=a\xi^\lambda\dd^\mu$ and $\ls_\preceq(v)=b\xi^\rho\dd^\sigma$
with $a,b\in K^\bucato\!$ and $\lambda,\mu,\rho,\sigma\in\NN_0^n$
such that $\lambda_i\ge\rho_i$ and $\mu_i\ge\sigma_i$ for all $i\in\{1,\ldots,n\}$,
we write $\frac{\ls_\preceq(u)}{\ls_\preceq(v)}$ for $ab^{-1}\xi^{\lambda-\rho}\dd^{\mu-\sigma}$.

Given any left ideal~$L$ of~$W$, we denote by $\LT_\preceq(L)$
the ideal $\langle \LT_\preceq(x) \mid x\in L^{\bucato}_{}\rangle$ of $K[X,Y]$.
\end{ntn}

\begin{rem}\label{papero}
Let ${\preceq}\in\NO(N)$ and put ${\leq}:=\phi(\preceq)$.
Since $\Phi$ is a $K$-isomorphism, one easily sees that 
$\LT_{\leq}(p)=\Phi(\lt_{\preceq}(\Phi^{-1}(p)))=\LT_{\preceq}(\Phi^{-1}(p))$ for all $p\in K[X,Y]$
and 
$\LT_{\leq}(\Phi(w))=\Phi(\lt_{\preceq}(w))=\LT_{\preceq}(w)$ for all $w\in W$.
\end{rem}

\begin{dfn}
Let $L$ be a left ideal of $W$ and let $\mathrm{\preceq}$ be a normal ordering of $W$.
According to $\text{\cite[p.~6]{Sai}}$, we say that
a finite subset $B$ of $L$
is a \concept{Gr\"obner basis} of $L$ with respect to~${\preceq}$
if $L=\sum_{b\in B}Wb$ and $\LT_\preceq(L)=\langle \LT_\preceq(b) \mid  b\in B^{\bucato}_{} \rangle$.
\end{dfn}


\begin{rem}\label{B U C}
Let $L$ be a left ideal of $W$ and ${\preceq}$ be a normal ordering of $W$.
Let $B$ be a Gr\"obner basis of $L$ with respect to ${\preceq}$.
Then for each finite subset $F$ of $L$ also $B \cup F$  clearly is a Gr\"obner basis of $L$ with respect to~${\preceq}$.
\end{rem}

\begin{thm}\label{GB exist}
Let $L$ be a left ideal of $W$ and $\mathrm{\preceq}$ be a normal ordering of $W$.
Then $L$ admits a Gr\"obner basis with respect to~$\mathrm{\preceq}$.
\end{thm}

\begin{proof}
Suppose that $L$ admits no Gr\"obner basis with respect to $\mathrm{\preceq}$.
Since $W$ is left noetherian, 
there exists a finite subset $F_0$ of $L$ such that $L=\sum_{f\in F_0}Wf$.
Consider the ideal $I_0:=\langle\LT_\preceq(f)\mid f\in F^{\bucato}_{0}\rangle$ of $K[X,Y]$.
It holds ${I_0\subsetneq\LT_\preceq(L)}$, because $F_0$ is not a Gr\"obner basis. 
Thus there exists $x_1\in L\wo \{0\}$ with ${\LT_\preceq(x_1)\notin I_0}$.
Put ${F_1:=F_0\cup \{x_1\}}$
and consider the ideal $I_1:=\langle\LT_\preceq(f)\mid f\in F^{\bucato}_{1}\rangle$ of $K[X,Y]$.
Again, it holds ${I_1\subsetneq\LT_\preceq(L)}$, because $F_1$ is not a Gr\"obner basis. 
Thus there exists $x_2\in L\wo \{0\}$ with ${\LT_\preceq(x_2)\notin I_1}$.
Put ${F_2:=F_1\cup \{x_2\}}$
and consider the ideal $I_2:=\langle\LT_\preceq(f)\mid f\in F^{\bucato}_{2}\rangle$  of $K[X,Y]$.
Going on in this manner we construct an infinite chain $I_0\subsetneq I_1\subsetneq I_2\subsetneq\ldots$
of ideals of $K[X,Y]$, in contradiction to the noetherianity of $K[X,Y]$.
\end{proof}

\begin{lem}\label{LTLT}
Let ${\preceq}\in\NO(N)$, ${\leq}:=\phi(\preceq)$.
For all $u,v\in W^\bucato$: 
\textnormal{(a)}~$\LT_\preceq(u+v)\leq\max_{\leq}\{\LT_\preceq(u),\LT_\preceq(v)\}$ whenever $u+v\ne 0$ with equality holding
if $\LT_\preceq(u)\ne\LT_\preceq(v)$,
\textnormal{(b)}~$\LT_\preceq(uv)=\LT_\preceq(u)\LT_\preceq(v)$,
and \textnormal{(c)}~$\LT_\preceq([u,v])<\LT_\preceq(u)\LT_\preceq(v)$ whenever $[u,v]\ne 0$.
\end{lem}

\begin{proof}
Statement~(a) is clear and follows from the inclusion $\Supp(u+v)\subseteq\Supp(u)\cup\Supp(v)$.
It also follows from the analogous result in $K[X,Y]$ because $\Phi$ is $K$-linear.

Since $M=\{\LT_\preceq(u)\LT_\preceq(v)\mid u,v\in W^\bucato\}$,
we may prove statements (b) and (c) by transfinite induction over $\LT_\preceq(u)\LT_\preceq(v)$
in the well-ordered set $(M,\leq)$.

Let $u,v\in W^\bucato\!$.
If $\LT_\preceq(u)\LT_\preceq(v)=1$, then $\LT_\preceq(u)=1=\LT_\preceq(v)$, hence $u\in K^\bucato$ and $v\in K^\bucato$,
so that (b) is clear  and (c) is trivially true as $[u,v]=0$.

Let $\LT_\preceq(u)\LT_\preceq(v)>1$ and assume that statements (b) and (c) hold for all $u',v'\in W^\bucato\!$ such that
$\LT_\preceq(u')\LT_\preceq(v')<\LT_\preceq(u)\LT_\preceq(v)$.

Choose any $(\lambda,\mu)\in \supp(u)$ and any $(\rho,\sigma)\in \supp(v)$.
If there exists $i\in\{1,\ldots,n\}$ such that ${\mu_i>0}$, 
we can write
$[\xi^\lambda\dd^\mu,\xi^\rho\dd^\sigma]
=\xi^\lambda\dd^{\mu-\epsilon_i}[\dd_i,\xi^\rho\dd^\sigma]+[\xi^\lambda\dd^{\mu-\epsilon_i},\xi^\rho\dd^\sigma]\dd_i$
with $\epsilon_i:=(\delta_{ih})_{1\le h\le n}$ where ${\delta_{ih}\in\NN_0}$ is the Kronecker delta.
Since $\dd_i$ and $\dd^\sigma$ commute, it holds $[\dd_i,\xi^\rho\dd^\sigma]=[\dd_i,\xi^\rho]\dd^\sigma$.
It follows that $[\dd_i,\xi^\rho\dd^\sigma]=0$ if $\rho_i=0$,
whereas $[\dd_i,\xi^\rho\dd^\sigma]=\rho_i\xi^{\rho-\epsilon_i}\dd^\sigma$ if $\rho_i>0$.
If $\rho_i>0$, then 
$\LT_\preceq(\xi^\lambda\dd^{\mu-\epsilon_i}[\dd_i,\xi^\rho\dd^\sigma])
=X^{\lambda+\rho-\epsilon_i}Y^{\mu+\sigma-\epsilon_i}$ by the induction hypothesis.
By the induction hypothesis, 
$\LT_\preceq([\xi^\lambda\dd^{\mu-\epsilon_i},\xi^\rho\dd^\sigma])<X^{\lambda+\rho}Y^{\mu+\sigma-\epsilon_i}$.
Thus $\LT_\preceq([\xi^\lambda\dd^{\mu-\epsilon_i},\xi^\rho\dd^\sigma])\LT_\preceq(\dd_i)<X^{\lambda+\rho}Y^{\mu+\sigma}$
and hence we may appeal again to the induction hypothesis to get 
$\LT_\preceq([\xi^\lambda\dd^{\mu-\epsilon_i},\xi^\rho\dd^\sigma]\dd_i)
=\LT_\preceq([\xi^\lambda\dd^{\mu-\epsilon_i},\xi^\rho\dd^\sigma])\LT_\preceq(\dd_i)
<X^{\lambda+\rho}Y^{\mu+\sigma}$.
We conclude by (a) that $\LT_\preceq([\xi^\lambda\dd^\mu,\xi^\rho\dd^\sigma])<X^{\lambda+\rho}Y^{\mu+\sigma}$.
Further, 
$\xi^\lambda\dd^\mu\xi^\rho\dd^\sigma=\xi^{\lambda+\rho}\dd^{\mu+\sigma}+\xi^\lambda[\dd^\mu,\xi^\rho]\dd^\sigma$.
Since $X^\rho Y^\mu \leq X^{\lambda+\rho}Y^{\mu+\sigma}$, one shows as above that
$\LT_\preceq([\dd^\mu,\xi^\rho])<X^\rho Y^\mu$.
Hence, using induction  and  compatibility twice,  
we get 
$\LT_\preceq(\xi^\lambda[\dd^\mu,\xi^\rho]\dd^\sigma)=\LT_\preceq(\xi^\lambda)\LT_\preceq([\dd^\mu,\xi^\rho])
\LT_\preceq(\dd^\sigma)<X^{\lambda+\rho}Y^{\mu+\sigma}$.
As clearly $\LT_\preceq(\xi^{\lambda+\rho}\dd^{\mu+\sigma})=X^{\lambda+\rho}Y^{\mu+\sigma}$,
it follows $\LT_\preceq(\xi^\lambda\dd^\mu\xi^\rho\dd^\sigma)=X^{\lambda+\rho}Y^{\mu+\sigma}$.

If $\mu=0$ and there exists $j\in\{1,\ldots,n\}$ such that $\sigma_j>0$, we reduce immediately to the previous case
since $[\xi^\lambda\dd^\mu,\xi^\rho\dd^\sigma]=-[\xi^\rho\dd^\sigma, \xi^\lambda\dd^\mu]$,
whereas if $\mu=0$ and $\sigma=0$, then  $[\xi^\lambda\dd^\mu,\xi^\rho\dd^\sigma]=0$
and clearly $\LT_\preceq(\xi^\lambda\dd^\mu\xi^\rho\dd^\sigma)=X^{\lambda+\rho}Y^{\mu+\sigma}$.

We write $u$ and $v$ in canonical form as
$u=\sum_{(\lambda,\mu)\in \supp(u)} a_{(\lambda,\mu)}\xi^\lambda\dd^\mu$ and
$v=\sum_{(\rho,\sigma)\in \supp(v)} b_{(\rho,\sigma)}\xi^\rho\dd^\sigma$ where
$a_{(\lambda,\mu)}\in K^\bucato$ for all $(\lambda,\mu)\in \supp(u)$ and 
$b_{(\rho,\sigma)}\in K^\bucato$ for all $(\rho,\sigma)\in \supp(v)$.
We find a unique $(\overline{\lambda},\overline{\mu})\in\supp(u)$ such that
$\lt_\preceq(u)=\xi^{\overline{\lambda}}\dd^{\overline{\mu}}$
and a unique $(\overline{\rho},\overline{\sigma})\in\supp(v)$ such that
$\lt_\preceq(v)=\xi^{\overline{\rho}}\dd^{\overline{\sigma}}$.
Thus $\LT_\preceq(u)\LT_\preceq(v)=X^{\overline{\lambda}+\overline{\rho}}Y^{\overline{\mu}+\overline{\sigma}}$.

If $(\lambda,\mu)\in\supp(u)\wo\{(\overline{\lambda},\overline{\mu})\}$,
say $\lambda\ne\overline{\lambda}$,
then $X^\lambda<X^{\overline{\lambda}}$.
Indeed, if $X^\lambda\ge X^{\overline{\lambda}}$,
then $X^\lambda Y^{\overline{\mu}}\ge X^{\overline{\lambda}}Y^{\overline{\mu}}$
by compatibilty,
thus $X^\lambda Y^{\overline{\mu}} =  X^{\overline{\lambda}}Y^{\overline{\mu}}$
as $X^{\overline{\lambda}} Y^{\overline{\mu}}=\LT_\preceq(u)$,
hence $\lambda=\overline{\lambda}$, a contradiction.
Similarly, $Y^\mu<Y^{\overline{\mu}}$ if $\mu\ne\overline{\mu}$.
Clearly, an analogous result holds for all $(\rho,\sigma)\in\supp(v)\wo\{(\overline{\rho},\overline{\sigma})\}$.
By compatibility it follows immediately that
$X^{\lambda+\rho}Y^{\mu+\sigma}<X^{\overline{\lambda}+\overline{\rho}}Y^{\overline{\mu}+\overline{\sigma}}$
for all $((\lambda,\mu),(\rho,\sigma))\in\supp(u)\times\supp(v)\wo
\{((\overline{\lambda},\overline{\mu}),(\overline{\rho},\overline{\sigma}))\}$.

It holds
$[u,v]
=
\sum_{((\lambda,\mu),(\rho,\sigma))\in\supp(u)\times\supp(v)}
a_{(\lambda,\mu)}b_{(\rho,\sigma)}[\xi^\lambda\dd^\mu,\xi^\rho\dd^\sigma]$.
By (a) and  the shown  inequalities
$\LT_\preceq([\xi^\lambda\dd^\mu,\xi^\rho\dd^\sigma])<X^{\lambda+\rho}Y^{\mu+\sigma}$
and 
$X^{\lambda+\rho}Y^{\mu+\sigma} \! \leq X^{\overline{\lambda}+\overline{\rho}}Y^{\overline{\mu}+\overline{\sigma}}$
for all $((\lambda,\mu),(\rho,\sigma))\in\supp(u)\times\supp(v)$,
we get
$\LT_\preceq([u,v])<X^{\overline{\lambda}+\overline{\rho}}Y^{\overline{\mu}+\overline{\sigma}}$.

It holds
$uv
=
a_{(\overline{\lambda},\overline{\mu})}b_{(\overline{\rho},\overline{\sigma})}
\xi^{\overline{\lambda}}\dd^{\overline{\mu}}\xi^{\overline{\rho}}\dd^{\overline{\sigma}}
+
\sum
_{((\lambda,\mu),(\rho,\sigma))\in \supp(u)\times \supp(v)
\wo\{((\overline{\lambda},\overline{\mu}),(\overline{\rho},\overline{\sigma}))\}}
a_{(\lambda,\mu)}b_{(\rho,\sigma)}\xi^\lambda\dd^\mu\xi^\rho\dd^\sigma
$.
By (a) and by the shown equalities
$\LT_\preceq(\xi^\lambda\dd^\mu\xi^\rho\dd^\sigma)=X^{\lambda+\rho}Y^{\mu+\sigma}$
for all $((\lambda,\mu),(\rho,\sigma))\in\supp(u)\times\supp(v)$
and because
$X^{\lambda+\rho}Y^{\mu+\sigma} < X^{\overline{\lambda}+\overline{\rho}}Y^{\overline{\mu}+\overline{\sigma}}$
for all $((\lambda,\mu),(\rho,\sigma))\in\supp(u)\times\supp(v)
\wo\{((\overline{\lambda},\overline{\mu}),(\overline{\rho},\overline{\sigma}))\}$,
we conclude that
$\LT_\preceq(uv)=X^{\overline{\lambda}+\overline{\rho}}Y^{\overline{\mu}+\overline{\sigma}}$.
\end{proof}

\begin{pro}\label{divisione}
Let $w\in W$. Let $F\subseteq W$ be finite, ${\preceq}\in\NO(N)$, and ${\leq}:=\phi(\preceq)$.
Then there exist ${r\in W}$  and $\smash{(q_f)_{f\in F}\in W^{\oplus F}}$ such that: 
{\textnormal{(a)}}~${w=\sum_{f\in F} q_f f + r}$,
{\textnormal{(b)}}~${\forall f\in F : f\ne 0 \Rightarrow \forall s\in\Supp(r) : \LT_{\preceq}(f)\nmid\Phi(s)}$,
and 
{\textnormal{(c)}}~${w\ne 0\Rightarrow \forall f\in F : q_f f\ne 0 \Rightarrow \LT_{\preceq}(q_f f)\leq\LT_{\preceq}(w)}$.
\end{pro}

\begin{proof}
If $w=0$, we put $r:=0$ and $(q_f)_{f\in F}:=(0)_{f\in F}$.
Let $w\ne 0$.
Since $M=\{\LT_\preceq(v)\mid v\in W^\bucato\}$,
we may proceed by transfinite induction in $(M,\leq)$ assuming that the statement holds for all $v\in W^\bucato\!$
such that ${\LT_{\preceq}(v)<\LT_{\preceq}(w)}$.
We distinguish between two cases:
(i)~if there exists ${f^{\bucato}}\!\in F^\bucato\!$ such that ${\LT_\preceq(f^\bucato)\mid\LT_\preceq(w)}$, then we put
$\smash{w':=w-\frac{\ls_\preceq(w)}{\ls_\preceq(f^\bucato)}f^\bucato}$;
(ii)~otherwise we set $w':=w-\ls_\preceq(w)$.

In the  case~(i) it holds 
$\LS_\preceq(\smash{\frac{\ls_\preceq(w)}{\ls_\preceq(f^\bucato)}f^\bucato})
=\LS_\preceq(\smash{\frac{\ls_\preceq(w)}{\ls_\preceq(f^\bucato)}})\LS_\preceq(f^\bucato)
=\smash{\frac{\LS_\preceq(w)}{\LS_\preceq(f^\bucato)}\LS_\preceq(f^\bucato)}
=\LS_\preceq(w)$ by~\ref{LTLT}(b).
Therefore, provided that $w'\ne 0$, by~\ref{LTLT}(a) we see that $\LT_\preceq(w')<\LT_\preceq(w)$.
This last relation
clearly holds also in the  case~(ii) when $w'\ne 0$.
Either by the induction hypothesis or by the preliminarly treated case when $w'=0$,
we find $r'\in W$ and $(q'_f)_{f\in F}\in W^{\oplus F}$ such that
 pro\-per\-ties (a), (b), (c) hold for $w'$ with respect to $r'$ and $(q'_f)_{f\in F}$.
In the  case~(i) we put $r:=r'$ and assign
$\smash{q_{f^\bucato}:=q'_{f^\bucato}\!+\frac{\ls_\preceq(w)}{\ls_\preceq(f^\bucato)}}$
and $q_f:=q'_f$ for all $f\in F\wo\{f^\bucato\}$.
In the  case~(ii) we set $r:=r'+\ls_\preceq(w)$ and $q_f:=q'_f$ for all $f\in F$.

We now verify that in either case  properties~(a), (b), (c) are fulfilled by~$r$ and~$(q_f)_{f\in F}$.
Property~(a) is clearly satisfied.
As for property~(b), in the  case~(i) we have $\Supp(r)=\Supp(r')$, so that the statement holds
either by the induction hypothesis or trivially when $w'=0$.
In the  case~(ii) we have ${\Supp(r)\subseteq\Supp(r')\cup\{\lt_\preceq(w)\}}$,
thus (b) holds
by the induction hypothesis, when $w'\ne 0$, 
and by our assumption that ${\LT_\preceq(f)\nmid\LT_\preceq(w)}$ for all $f\in {F^\bucato}\!$.

Let us focus on property~(c).
In the  case~(i),
when $w'=0$, then $q_f=0$ for all $f\in F\wo\{f^\bucato\}$ and $q_{f^\bucato}=\frac{\ls_\preceq(w)}{\ls_\preceq(f^\bucato)}$,
so that $q_{f^\bucato}f^\bucato=w$ and hence $\LT_\preceq(q_{f^\bucato}f^\bucato)=\LT_\preceq(w)$.
When $w'\ne 0$, by the induction hypothesis and by what we have said afore,
for all $f\in F\wo\{f^\bucato\}$ with $q_f f\ne 0$
we obtain $\LT_\preceq(q_f f)=\LT_\preceq(q'_f f)\leq\LT_\preceq(w')<\LT_\preceq(w)$,
while as for $f^\bucato\!$,
whenever $q_{f^\bucato}f^\bucato\ne\!0$,
using in addition~\ref{LTLT}
we get
$\LT_\preceq(q_{f^\bucato} f^\bucato)
\leq\max_\leq\{ \LT_\preceq(q'_{f^\bucato}f^\bucato),
                \LT_\preceq(\smash{\frac{\ls_\preceq(w)}{\ls_\preceq(f^\bucato)}f^\bucato}) \}
\leq\max_\leq\{ \LT_\preceq(w'), \LT_\preceq(w) \}
=\LT_\preceq(w)$
if $q'_{f^\bucato}f^\bucato\ne 0$,
and similarly 
$\LT_\preceq(q_{f^\bucato}f^\bucato)
=\LT_\preceq(\smash{\frac{\ls_\preceq(w)}{\ls_\preceq(f^\bucato)}f^\bucato})
=\LT_\preceq(w)$
if $q'_{f^\bucato}f^\bucato=0$.

In the case~(ii), when $w'=0$, then $q_f=0$ for all $f\in F$, so that (c) holds trivially.
When $w'\ne 0$, then  by the induction hypothesis  we have
$\LT_\preceq(q_f f)=\LT_\preceq(q'_f f)\leq\LT_\preceq(w')<\LT_\preceq(w)$ whenever $q_ff\ne 0$.
\end{proof}

\begin{pro}\label{restriction}
Let $L$ be an ideal of $W$, let ${\preceq}$ be a normal ordering of $W$ and let $B$ be a Gr\"obner basis of $L$
with respect to ${\preceq}$.
Let ${\preceq}'$ be a normal ordering of $W$ such that
${\preceq}'\res_{\Supp(B)}={\preceq}\res_{\Supp(B)}$.
Then $B$ is a Gr\"obner basis of $L$ with respect to ${\preceq}'$.
\end{pro}

\begin{proof}
Let ${\leq}:=\phi(\preceq)$ and ${\leq}':=\phi(\preceq')$ be the induced monomial orderings of $K[X,Y]$.
Let $x\in L^\bucato$.
In view of~\ref{divisione},
we can write $x = \sum_{b\in B} q_b b + r$ for some  ${r\in W}$ and some ${(q_b)_{b\in B}\in W^{\oplus B}}$
enjoying the properties: 
(i)~${\LT_{\preceq'}(q_b b)\leq'\LT_{\preceq'}(x)}$ whenever ${q_b b\ne 0}$,
and
(ii)~${\LT_{\preceq'}(b)\nmid \Phi(s)}$ for all ${s\in\Supp(r)}$ whenever ${b\ne 0}$.

Clearly, $r\in L$.
Suppose that ${r\ne 0}$.
Then $\LT_\preceq(r)\in\LT_\preceq(L)$, thus the monomial $\LT_\preceq(r)$ lies in the monomial ideal 
$\langle\LT_\preceq(b)\mid b\in B^\bucato\rangle$ of $K[X,Y]$.
Hence there exists $b\in B^\bucato$ such that $\LT_\preceq(b)\mid\LT_\preceq(r)$, see~\cite[Lemma~2.4.2]{Cox}.
Since ${\preceq}$ and ${\preceq'}$ agree on $\Supp(B)$,
we have $\LT_{\preceq}(b)=\LT_{\preceq'}(b)$, and it follows
${\LT_{\preceq'}(b)\mid\LT_{\preceq}(r)\in\Phi(\Supp(r))}$,
in contradiction to~(ii).

Hence $r=0$.
So it follows from~(i) that $\Phi(x)=\sum_{b\in B}\Phi(q_b b)$ with $\LT_{\leq'}(\Phi(q_b b)) \leq' \LT_{\leq'}(\Phi(x))$
whenever $q_b b\ne 0$.
Thus there exists $b'\in B$ with $q_{b'}b'\!\ne 0$ such that $\LT_{\leq'}(\Phi(x))=\LT_{\leq'}(\Phi(q_{b'}b'))$,
i.e.~$\LT_{\preceq'}(x)=\LT_{\preceq'}(q_{b'} b')$.
We get $\LT_{\preceq'}(x)=\LT_{\preceq'}(q_{b'})\LT_{\preceq'}(b')$ by~\ref{LTLT}(b),
so
$\LT_{\preceq'}(x)\in\langle\LT_{\preceq'}(b')\rangle$.

We have shown that
$\LT_{\preceq'}(L)=\langle\LT_{\preceq'}(b)\mid b\in B^\bucato\rangle$.
Since clearly $L=\sum_{b\in B}Wb$, we conclude that $B$ is a Gr\"obner basis of $L$ with respect to~${\preceq'}$.
\end{proof}

\begin{lem}\label{open}
Let $L$ be a left ideal of $W$ and $F$ be a finite subset of $L$.
Then the set $\mathfrak{V}_L(F)$ of all normal orderings $\mathrm{\preceq}$ of $W$ such that $F$ is a Gr\"obner basis of $L$
with respect to $\mathrm{\preceq}$ is open in $\NO(N)$.
\end{lem}

\begin{proof}
Without restriction we assume that $\mathfrak{V}_L(F)\ne\emptyset$.
Let $\mathrm{\preceq}\in \mathfrak{V}_L(F)$.
So $F$ is a Gr\"obner basis of $L$ with respect to $\mathrm{\preceq}$.
Let $(S_i)_{i\in\NN_0}$ be a filtration of $N$ consisting of finite sets $S_i$.
We find $r\in\NN_0$ such that 
the finite subset $\Supp(F)$ of $N$ lies in $S_{r+1}$.
In the notation of~\ref{U=N},
consider the open neighbourhood ${\mathfrak{N}_r(\preceq)\cap\NO(N)}$
of $\mathrm{\preceq}$ in $\NO(N)$
and let ${\mathrm{\preceq'}\in {\mathfrak{N}_r(\preceq)\cap\NO(N)}}$.
So $\mathrm{\preceq'}$ and $\mathrm{\preceq}$ agree on~$S_{r+1}$ and in particular on $\Supp(F)$.
From~\ref{restriction} it follows that $F$ is a Gr\"obner basis of $L$ with respect to $\mathrm{\preceq}'$,
thus $\mathrm{\preceq}'\in \mathfrak{V}_L(F)$.
Therefore ${\mathfrak{N}_r(\preceq)\cap\NO(N)}\subseteq \mathfrak{V}_L(F)$, and hence $\mathfrak{V}_L(F)$ is open in $\NO(N)$.
\end{proof}

\begin{rem}\label{covering}
Let $L$ be a left ideal of $W$.
For each $\mathrm{\preceq}\in\NO(N)$ we can choose a Gr\"obner basis $B_\preceq$ of $L$ with respect to $\mathrm{\preceq}$
by~\ref{GB exist}.
Of course $\mathrm{\preceq}\in \mathfrak{V}_L(B_{\preceq})$.
Hence $(\mathfrak{V}_L(B_\preceq))_{\mathrm{\preceq}\in\NO(N)}$ is an open covering of $\NO(N)$ by~\ref{open}.
\end{rem}

\begin{dfn}
Let $L$ be a left ideal of $W$.
A finite subset $V$ of $L$ is a \concept{universal Gr\"obner basis} of $L$ if $V$ is a Gr\"obner basis of $L$
with respect to each normal ordering $\mathrm{\preceq}$ of $W$.
\end{dfn}

\begin{thm}\label{UGB exist}
Each left ideal $L$ of $W$ admits a universal Gr\"obner basis.
\end{thm}

\begin{proof}
By~\ref{covering} we can choose an open covering $(\mathfrak{V}_L(B_\preceq))_{\mathrm{\preceq}\in\NO(N)}$
of $\NO(N)$ where each $B_\preceq$ is a Gr\"obner basis of $L$ with respect to~$\mathrm{\preceq}$.
Since $\NO(N)$ is compact, see~\ref{NO(N)},
we find a finite subcovering $(\mathfrak{V}_L(B_{\preceq_k}))_{1\le k\le t}$
with $t\in\NN$.
We claim that $V:=\bigcup_{1\le k\le t} B_{\preceq_k}$ is a universal Gr\"obner basis of~$L$.
Indeed, let $\mathrm{\preceq}_0\in\NO(N)$.
As $\NO(N)=\bigcup_{1\le k\le t}\mathfrak{V}_L(B_{\preceq_k})$, we have $\mathrm{\preceq}_0\in\mathfrak{V}_L(B_{\preceq_k})$
for some $k\in\{1,\ldots,t\}$.
Thus $B_{\preceq_k}$ is a Gr\"obner basis of $L$ with respect to $\mathrm{\preceq}_0$.
From~\ref{B U C} it follows that $V$ is a Gr\"obner basis of $L$ with respect to~$\mathrm{\preceq}_0$.
As the choice of ${\preceq}_0$ in $\NO(N)$ was arbitrary, we conclude that $V$ is a universal Gr\"obner basis of~$L$.
\end{proof}


\vskip\baselineskip

\renewcommand{\bibname}{REFERENCES}
\renewcommand{\refname}{REFERENCES}


\begin{thebibliography}{99}
%
\bibitem{Bor}
A.~Borel \& al.,
\textbf{Algebraic $D$-Modules},
Perspectives in Mathematics 2,
Academic Press,
1987.
%
\bibitem{Cou}
S.~C.~Coutinho,
\textbf{A Primer of Algebraic $D$-Modules},
London Mathematical Society Student Texts 33,
Cambridge University Press,
1995.
%
\bibitem{Cox}
D.~Cox \&
J.~Little \&
D.~O'Shea,
\textbf{Ideals, Varieties, and Algorithms},
Undergraduate Texts in Mathematics, 
Springer, 
2007 \& 1997 \& 1992.
%
%
\bibitem{Sai}
M.~Saito \&
B.~Sturmfels \&
N.~Takayama,
\textbf{Gr\"obner Deformations of Hypergeometric Differential Equations},
Algorithms and Computation in Mathematics 6,
Springer,
2000.
%
\bibitem{Sik}
A.~S.~Sikora,
\textbf{Topology on the Spaces of Orderings of Groups},
Bulletin of the London Mathematical Society 36 (2004), 519--526.
%
\bibitem{Wat}
J.~J.~Watkins,
\textbf{Topics in Commutative Ring Theory},
Princeton University Press, 2007.
\end{thebibliography}
\end{document}